\documentclass[12pt]{amsart}
\usepackage[T1]{fontenc}
\usepackage[utf8]{inputenc}
\usepackage[margin=1in]{geometry}
\usepackage{mdframed}
\usepackage[colorinlistoftodos]{todonotes}
\usepackage{tikz}
\usetikzlibrary{matrix}
\usepackage{float}
\usepackage{amssymb}
\usepackage{amsmath}
\usepackage[title]{appendix}
\usepackage{amsthm}
\usepackage{amsfonts}
\usepackage{tikz-cd}
\usepackage{comment}
\usepackage{mathtools}
\usepackage[colorlinks=true,
            linkcolor=black,
            citecolor=black,
            urlcolor=black]{hyperref}
\usepackage{mathrsfs}
\usepackage{enumitem}
\usepackage{tabularx}
\usepackage{todonotes}
\DeclarePairedDelimiter\abs{\lvert}{\rvert}%
\newtheorem{theorem}{Theorem}[section]
\newtheorem{lemma}[theorem]{Lemma}
\newtheorem{remark}[theorem]{Remark}

\newtheorem{definition}[theorem]{Definition}
\newtheorem{example}[theorem]{Example}
\newtheorem{proposition}[theorem]{Proposition}
\newtheorem{notation}[theorem]{Notation}
\newtheorem{question}[theorem]{Question}
\newtheorem{corollary}[theorem]{Corollary}
\newtheorem{recollection}[theorem]{Recollection}

\theoremstyle{plain} 
\newcommand{\thistheoremname}{}
\newtheorem{genericthm}[theorem]{\thistheoremname}

\newtheorem*{genericthm*}{\thistheoremname}
\newenvironment{namedthm*}[1]
  {\renewcommand{\thistheoremname}{#1}%
   \begin{genericthm*}}
  {\end{genericthm*}}
\newcommand{\Z}{\mathbb{Z}}

\newcommand{\C}{\mathbb{C}}

\newcommand{\E}{\mathbb{E}}
\renewcommand{\S}{\mathbb{S}}
\renewcommand{\P}{\mathbb{P}}
\newcommand{\F}{\mathbb{F}}

\DeclareMathOperator{\DA}{\mathcal{DA}}

\DeclareMathOperator{\id}{\text{id}}

\DeclareMathOperator{\Mod}{\text{Mod}}

\DeclareMathOperator{\fil}{\mathrm{fil}}
\DeclareMathOperator{\G}{\mathbb{G}}

\DeclareMathOperator{\NAlg}{NAlg}
\DeclareMathOperator{\Vect}{Vect}
\DeclareMathOperator{\ev}{\mathrm{ev}}

\DeclareMathOperator{\Sp}{Sp}

\DeclareMathOperator{\MU}{\mathrm{MU}}

\DeclareMathOperator{\NA}{\mathcal{NA}}

\DeclareMathOperator{\BGL}{\text{BGL}}
\DeclareMathOperator{\BU}{\mathrm{BU}}

\DeclareMathOperator{\Alg}{\mathrm{Alg}}

\DeclareMathOperator{\BSU}{\mathrm{BSU}}

\DeclareMathOperator{\Tor}{\mathrm{Tor}}
\DeclareMathOperator{\MGL}{\mathrm{MGL}}
\DeclareMathOperator{\Fil}{\mathrm{Fil}}
\DeclareMathOperator{\cell}{\mathrm{cell}}
\DeclareMathOperator{\SH}{\mathrm{SH}}
\DeclareMathOperator{\THH}{\mathrm{THH}}

\DeclareMathOperator{\Spc}{\mathrm{Spc}}

\DeclareMathOperator{\BSL}{\mathrm{BSL}}
\DeclareMathOperator{\Pic}{\mathrm{Pic}}

\title[$\mathbb{C}$-Motivic Eilenberg-Mac Lane Spectra as $\mathrm{MGL}$-Motivic Thom Spectra and $\mathbb{G}_{m} \odot {-}$]{$\mathbb{C}$-Motivic Eilenberg-Mac Lane Spectra as $\mathrm{MGL}$-Motivic Thom Spectra and Normed $\mathbb{G}_m$-Hochschild Homology}
\author{Preston Cranford}

\date{date: 2026-Oct-19}
\begin{document}
\begin{abstract}
  We exhibit the $p$-complete $\C$-motivic Eilenberg-Mac Lane spectra $\mathrm{M}\mathbb{F}_p$ and $(\mathrm{M}\mathbb{Z})_p^{\wedge}$ as $\E_3$-$\MGL_p^{\wedge}$-motivic Thom spectra over the motivic $\G_m$-loop spaces $\Omega_{\mathbb{G}_m}\BGL$ and $\Omega_{\mathbb{G}_m}\BSL$, respectively. Furthermore, for $p=2$, we compute the $\MGL_2^{\wedge}$-relative $\G_m$-Hochschild Homology $\G_m \odot^{\MGL_2^{\wedge}} {-}$, recently introduced by Bachmann-Burklund-Dundas-{\O}stv{\ae}r, of $\mathrm{M}\mathbb{F}_2$ and $(\mathrm{M}\mathbb{Z})_2^{\wedge}$ to be $\mathrm{M}\mathbb{F}_2[\BGL]$ and $(\mathrm{M}\mathbb{Z})_2^{\wedge}[\mathrm{BSL}]$ as $\E_2$-$\mathrm{M}\mathbb{F}_2$-algebras and $\E_2$-$(\mathrm{M}\mathbb{Z})_2^{\wedge}$-algebras, respectively.
\end{abstract}
\maketitle
\tableofcontents
\section{Introduction}
Motivic homotopy theory is about the emulation of homotopy theory in algebraic geometry. To the author's knowledge, the phenomenom of the Hopkins-Mahowald theorem from homotopy theory hasn't been well emulated in algebraic geometry. The Hopkins-Mahowald theorem, at the prime $p=2$ is equivalent in math content to each of the four following statements statements"
\begin{itemize}
	\item(Hopkins-Mahowald Theorem, \cite{mahowaldRingSpectraWhich1979}) There exists a spherical bundle on the topological space $\Omega^2 S^3$ whose Thom spectrum is isomorphic to $\F_2$ as $\E_2$-rings.
	\item (Steinberger's calcualtion, \cite{brunermaymccluresteinbergerEnsuremathH_inftyRingSpectra1986}) The dual Steenrod algebra $\F_2 \otimes \F_2$ is isomorhic to the $\F_2$ homology $\F_2[\Omega^2 S^3]$ of $\Omega^2 S^3$ as $\E_2$-rings.
	\item (B\"okstedt periodicity, \cite{bokstedtTopologicalHochschildHomology1985}) The topological Hochschild Homology $\THH(\F_2)$ of $\F_2$ is isomorphic to $\F_2[\Omega S^3]$ as $\E_1$-$\F_2$-algebras.
	\item (versal characteristic algebra, \cite{antolin-camarenabarthelSimpleUniversalProperty2019}) There exists a $\E_2$-ring with universal property of being the free $\E_2$-ring witnessing a nullhomotopy $2=0$, and that $\E_2$-ring is isomorphic to $\F_2$.
\end{itemize}

There have been attempts in literature to understand various parts of these equivalent statements in motivic settings. The idea is to try to prove or disprove the following question \ref{q} which we think of as asking about the existance of a ``motivic Hopkins-Mahowald theorem''.
\begin{question}\label{q}
	Fix a field $k$. Does there exist a motivic space $X \in \Spc(k)$ and a motivic spherical bundle on $X$ whose associated motivic Thom spectrum in $\SH(k)$ is isomorphic to the motivic Eilenberg-Mac Lane spectrum $\mathrm{M}\F_2$?
\end{question}

In \cite{dundasHochschildHomologyModp2024}, Dundas-Hill-Ormsby-Østvær studied a version of (topological) Hochschild homology they call motivic Hochschild Homology. It is contructed by tensoring with the motivic circle $S^1=S^{1,1}$ in commutative algebras in $\SH(k)$. They then compute motivic Hochschild homology of $\mathrm{M}\F_2$ and get a motivic-Hochschild-Homology-of-$\mathrm{M}\F_2$ calculation containing $\tau$-torsion.

In \cite{bachmannMotivicHochschildHomology2026}, Bachmann-Burklund-Dundas-Østvær introduce another notion of motivic Hochschild homology $\G_m \odot {-}$ that differs from the above notion in two ways. Firstly, the tensor is formed in normed motivic spectra \cite{bachmannNormsMotivicHomotopy2021} rather than commutative algebras in motivic spectra. Secondly, the tensor is taken with the motivic space $\G_m=S^{1,0}$ rather than $S^1=S^{1,1}$. Here the authors compute the motivic Hochschild $\G_m \odot \mathrm{M}\F_2$ of $\mathrm{M}\F_2$ and show that it is the free $\E_1$-$\mathrm{M}\F_2$-algebra $\mathrm{M}\F_2[\mu_{2,1}]$ on a class $\mu_{2,1}$ in bidegre $(2,1)$. This computation of motivic Hochschild homology of $\mathrm{M}\F_2$ looks more like classical B\"okstedt periodicity having no $\tau$-torsion.

Recent work \cite{cranfordGeneralizedThomSpectra2026} of the author has produced $\MU$-relative verions of the Hopkins-Mahowald theorem and B\"oksdedt periodicity. These statements seem much more amendable to motivic lifting methods than the do the classical Hopkins-Mahowald theorem and B\"Okstedt periodicity statement. For example, \cite[Theorem A]{cranfordGeneralizedThomSpectra2026} states that $\F_2$ is the isomorphic to the an $\MU$-line bundle on the space $U$ as $\E_3$-$\MU$-algebras. The motivic lifts of (the classifying space of the base space $U$) $BU$ and (coefficent ring spectrmu of line bundes) $\MU$, motivically refine to the the motivic space $\BGL$ and the motivic spectrum $\MGL$. Motivic objects built out of these enjoy niceness properties among being Tate, cellular, and pure. For these reasons, we are able to record a $\MGL$-relative motivic analogue of the Hopkins-Mahowald theorem as Theorem A while we remain unable to answer, postiviely or negatively, Question \ref{q}. 

\begin{namedthm*}{Theorem A}[Theorem \ref{thm:A}]
	For every prime $p$, there exists an $\E_3$-monoidal $p$-complete $\C$-motivic space map
	\[
		\Omega_{\G_m}\BGL\longrightarrow\BGL_1(\MGL_p^{\wedge})
	\]
	whose associated $\E_3$-$\MGL_p^{\wedge}$-motivic Thom spectrum is isomorphic to $\mathrm{M}\F_p$. Precomposition by the canonical map $\Omega_{\G_m}\BSL\rightarrow\Omega_{\G_m}\BGL$ yields an $\E_3$-monoidal $p$-complete $\C$-motivic space map
	\[
		\Omega_{\G_m}\BSL\longrightarrow\BGL_1(\MGL_p^{\wedge})
	\]
	whose associated $\E_3$-$\MGL_p^{\wedge}$-motivic Thom spectrum is isomorphic to $(\mathrm{M}\Z)_p^{\wedge}$.
\end{namedthm*}

As a corollary we record motivic Thom isomoprhisms associated to Theorem A.

\begin{corollary}
	The first motivic Thom spectrum statement of Theorem A has associated Thom isomoprhism
	\[
		\mathrm{M}\F_p \otimes_{\MGL_{p}^{\wedge}} \mathrm{M}\F_p
		\simeq
		\mathrm{M}\F_p[\Omega_{\G_m} \BGL]
	\]
	of $\E_3$-$\mathrm{M}\F_p$-algebras. The second motivic Thom spectrum statement of Theorem A has associated Thom isomoprhism
	\[
		(\mathrm{M}\Z)_p^{\wedge} \otimes_{\MGL_{p}^{\wedge}} (\mathrm{M}\Z)_p^{\wedge}
		\simeq
		(\mathrm{M}\Z)_p^{\wedge}[\Omega_{\G_m} \BSL]
	\]
	of $\E_3$-$(\mathrm{M}\Z)_p^{\wedge}$-algebras.
\end{corollary}

As a second corllary to Theorem A, we remark on the Hopkins-Hoyois-Morel theorem \cite{hoyoisAlgebraicCobordismMotivic2015}. It is the motivic analogue of the Lazard ring calculation. We record the following $\E_3$-monoidal $\C$-moivic analogue of the Hopkins-Hoyois-Morel theorem.

\begin{corollary}[Corollary \ref{cor:HHM}]\label{intro:HHM}
	Let $a_i$ denote the generators of the Lazard ring $(\MGL_{p, \ast,\ast}^{\wedge})_0\cong \Z_p^{\wedge}[a_1, a_2, \dots]$ where $\deg(a_i)=(2i,i)$. Then there exists an equivalence
	\[
		\MGL_p^{\wedge}/(a_1,a_2,\dots)\simeq(\mathrm{M}\Z)_p^{\wedge}
	\]
	of $\E_3$-$\MGL_p^{\wedge}$-algebras. 
	Furthermore there exists an equivalence
	\[
		\MGL_p^{\wedge}/(p, a_1,a_2,\dots)\simeq\mathrm{M}\F_p
	\]
	of $\E_3$-$\MGL_p^{\wedge}$-algebras.
\end{corollary}

\begin{remark}
	We note that our proof of Corollary \label{intro:HHM} itself depends on the Hopkins-Hoyois-Morel theorem.
\end{remark}

For our other main theorem, Theorem B, for $p=2$, we record $\G_m \odot {-}$ assoicated to the motivic Thom spectra of Theorem A. We warn that $\G_m \odot {-}$ are defined only for normed motivic spectra and the Thom spectrum statements of theorem A are not stated in $\E_{\infty}$ form and thus cannot be interpreted in a normed way. Instead, we utilize the corresponding Thom isomoprhisms on the level of modules and largely follow the work of \cite{bachmannMotivicHochschildHomology2026} to prove Theorem B. We make sense of $\MGL_p^{\wedge}$-relative $\G_m \odot^{\MGL} {-}$ by utilizing maps of normed spectra $\MGL_p^{\wedge}\rightarrow {-}$ as studied in \cite{bachmannNormsMotivicHomotopy2021}.

\begin{namedthm*}{Theorem B}
	There exists an equivalence
	\[
		\G_m \odot^{\MGL_2^{\wedge}}\mathrm{M}\F_2
		\simeq
		\mathrm{M}\F_2[\BGL]
	\]
	of $\E_2$-$\mathrm{M}\F_2$-algberas. Furthermore there exists an equivalence
	\[
		\G_m \odot^{\MGL_2^{\wedge}}\mathrm{M}\Z_2^{\wedge}
		\simeq
		(\mathrm{M}\Z)_2^{\wedge}[\BSL]
	\]
	of $\E_2$-$(\mathrm{M}\Z)_2^{\wedge}$-algberas.
\end{namedthm*}

Whereas the main theorem of \cite{bachmannMotivicHochschildHomology2026} shows that $\G_m \odot \mathrm{M}\F_2$ is the free $\E_1$-$\mathrm{M}\F_2$-algebra $\mathrm{M}\F_2[\mu_{2,1}]$ on a single generator $\mu_{2,1}$, we conclude the following corollary of Theorem B.

\begin{corollary}
	We have that $\G_m \odot^{\MGL_2^{\wedge}} \mathrm{M}\F_2$ is the product
	\[
		\mathrm{M}\F_2[\mu_{2,1}] \otimes
		\mathrm{M}\F_2[\mu_{4,2}] \otimes
		\mathrm{M}\F_2[\mu_{6,3}] \otimes
			\cdots
	\]
	of free $\E_1$-$\mathrm{M}\F_2$-algebras $\mathrm{M}\F_2[\mu_{2n,n}]$ on classes $\mu_{2n,2}$ of bidegree $(2n,n)$. LIkewise $\G_m \odot^{\MGL_2^{\wedge}} (\mathrm{M}\Z)_2^{\wedge}$ is the product
	\[
		\mathrm{M}\F_2[\mu_{4,2}] \otimes
		\mathrm{M}\F_2[\mu_{6,3}] \otimes
		\mathrm{M}\F_2[\mu_{8,4}] \otimes
		\cdots
	\]
	of free $\E_1$-$\mathrm{M}\F_2$-algebras $\mathrm{M}\F_2[\mu_{2n,n}]$ on classes $\mu_{2n,2}$ of bidegree $(2n,n)$, starting with $(2n,n)=(4,2)$.
\end{corollary}

\begin{remark}
	Our proof of Theorem B, like the proof of the main theorem of \cite{bachmannMotivicHochschildHomology2026}, depends on forthcoming work of Bachmann-Hopkins regarding $\mathrm{M}\F_2$ and $(\mathrm{M}\Z)_2^{\wedge}$ having normed motivic structured after modding out by $\tau$.
\end{remark}

\subsection*{Conventions}
We say category to mean $\infty$-category=$(\infty,1)$-category. For the remainder of this paper we fix base field $\C$. We refer to the motivic spectra category as $\SH(\C)$ and the motivic space category as $\Spc(\C)$. We refer to the classifying space construction as $B_s$ to emphasize its simplicial nature. We refer to the $\G_m$-loop space, $S^1$-loop space, and $\P^1$-loop space as $\Omega_{\G_m}$, $\Omega_{S^1}$, and $\Omega_{\P^1}$, respectively. When we write $\F_p$, we either mean the discrete commutative ring or the Eilenberg-Mac Lane spectrum whereas we (try to) always write $\mathrm{M}\F_p$ for the motivic Eilenberg-Mac Lane spectrum. In sections \ref{sec:math} and \ref{sec:MHH} we always $p$-complete everything, and in section \ref{sec:MHH}, $p=2$ always. We warn that we use $\BGL$ to both denote the motivic space $\BGL$ as well as the space $\BGL \subset \Pic$ that is the connected component of $\Pic$ containing the identity. We write $\BSL$ to denote the motivic space $\BSL$ and well as write $K$ to denote the algebraic K-theory motivic space $\Z \times \BGL$. We write $\mathrm{Re}$ to denote Betti realization.

\subsection*{Acknowledgements}
The author has benefited from conversation on this paper  with Sanath Devalapurkar. This work was supported by the National Science Foundation grant DGE-2234667 and Simons Foundation grant SDF-00014770.

\subsection*{Large Language Model Useage Disclosure}
The author used several OpenAI LLM models, including GPT-Sol 6.1, for feedback on writing, both mathematically and expositionally. Some proofs in this paper reflect feedback from the LLM model. Any LLM feedback that has influenced this paper has been only carefully and skeptically considered. 
The writing in this paper is entirely human generated, and the large scale mathematical ideas of this paper were entirely human generated.  The author maintains responsibility for the accuracy of the entirety of this paper.

\section{Generalized $\C$-Motivic Thom Spectra}\label{sec:math}
The purpose of this section is to prove Theorem A concerning motivic Thom spectra. We begin by recalling facts.
\begin{recollection}
The functor
\[
	\fil_{\ev}\colon
	\Alg_{\E_n}\left(\Mod_{\fil_{\ev}(\S)}\Fil\Sp\right)
	\rightarrow
	\Alg_{\E_n}\left(\SH(\C)^{\cell}\right)
\]
given by Pstragowski's even filtration \cite{pstragowskiPerfectEvenModules2023} induces an equivalence after $p$-completion between even $\MU$-synthetic spectra, modeled by modules filtered spectra on $\fil_{\ev}^{\ast}(\S)$ and cellular motivic spectra.
\end{recollection}

The following fact is standard.
\begin{lemma}
	Let $\mathcal{C}$ denote a presentably symmetric monoidal category, $n>0$,  and $f:X\rightarrow \Pic(\mathcal{C})$ denote an $\E_n$-monoidal space map where $X$ is grouplike and connected. Then the Thom object $T(f)$ associated to the colimit of $f$ identifies, as $\E_{n}$-algebras in $\mathcal{C}$, with
	\[
		\id_{\mathcal{C}} \otimes_{\id_{\mathcal{C}[\Omega X]}} \id_{\mathcal{C}}.
	\]
\end{lemma}

From \cite{cranfordGeneralizedThomSpectra2026} we recall there there exists an $\E_3$-space map $U\rightarrow \BGL_1(\MU)$ whose $\E_3$-$\MU$-Thom spectrum is isomorphic to $\F_2$ and furthermore precomposing with $SU \hookrightarrow U$ yields an $\E_3$-$\MU$-space map $SU\rightarrow \BGL_1(\MU)$ whose $\E_3$-$\MU$-Thom spectrum is isomorphic to $\Z$.

\begin{lemma}\label{lem:realize}
	Suppose $X$ is a motivic space whose motivic suspension spectrum $\Sigma_+^{\infty}X$ is an even finite $\MGL$-projective motivic spectrum. Then the natural map
	\[
		\MGL[X]
		\rightarrow
		\fil_{\ev}(\MU[\mathrm{Re}(X)])
	\]
	is an equivalence of $\E_{\infty}$-$\MGL$-algebras.
\end{lemma}

\begin{proof}
	By the finite $\MGL$-projectiveness assumption on $\Sigma_+^{\infty}X$ we have that the natural map
	\[
		(\MGL_{\ast, \ast}[X]\rightarrow \MU_{\ast}\mathrm{Re}(X)
	\]
	is an equivalence of graded $\MU_{\ast}$-modules by \cite[Lemma 7.9]{pstragowskiSyntheticSpectraCellular2023}. The claim follows from eveness.
\end{proof}

\begin{lemma}\label{lem:Tor_even}
	Suppose $X$ is grouplike pointed $\E_n$-space such whose motivic suspension spectrum $\fil_{\ev}\Sigma_+^{\infty}\Omega X$ is an even finite $\MGL$-projective motivic spectrum. and $f\colon X\rightarrow \BGL_1(\MU)$ is an $\E_n$-space map whose associated $\E_n$-$\MU$-Thom spectrum we denote $T(f)$. Then the natural map
	\[
		\fil_{\ev}(\MU) 
		\otimes_{\fil_{\ev}(\MU[\Omega X])}
		\fil_{\ev}(\MU)
		\rightarrow
		\fil_{\ev}\left(
			\MU \otimes_{\MU[\Omega X]} \MU
		\right)
	\]
	is an equivalence of $\E_n$-$\fil_{\ev}(\MU)$-algebras.
\end{lemma}

\begin{proof}
	This associated trigraded Tor sepctral sequence with signature
	\[
		E^2_{s,t,w}
		=
		\Tor_s^{\pi_{\ast, \ast}(\fil_{\ev}(\MU[\Omega X]))}
		\left(
			\pi_{\ast, \ast}(\fil_{\ev}(\MU)), \pi_{\ast, \ast}(\fil_{\ev}(\MU)))
		\right)_{t,w}
		\Rightarrow
		\pi_{s+t, w}
		\left(
			\fil_{\ev}\left(
			\MU \otimes_{\MU[\Omega X]} \MU
		\right)
		\right)
	\]
	whose convergence is atrributed to cellularity. This spectral sequence collapes on the $E^2$ page for eveness degree reasons and there are no extension problems to solve for filtration reasons.
\end{proof}

\begin{lemma}\label{lem:loop_spaces}
	The natural maps motivic spaces
	\[
		B_sK
		\rightarrow
		\Omega_{\G_m}\BGL
	\]
	and
	\[
		B_s \BGL
		\rightarrow 
		\Omega_{\G_m} \BSL
	\]
	that are equivalences of motivic spaces.
\end{lemma}

\begin{proof}
	By motivic Bott periodicity \cite[Proposition 4.1]{gepnerMotivicSpectraRepresenting2009}, the natural map
	\[
		K
		\rightarrow
		\Omega_{\P^1}K
	\]
	is an equivalence. By observing $\Omega_{\P^1}K=\Omega_{s}\Omega_{\G_m}K$ and applying $B_s$, one obtains the equivalence
	\[
		B_sK\rightarrow B_s\Omega_{s}\Omega_{\G_m} K
	\]
	from which one concludes the equivalence
	\[
		B_s K \rightarrow \Omega_{\G_m} K.
	\]
	The desired first statement follows from considering the fiber sequnce 
	\[
		\BGL
		\rightarrow 
		K 
		\rightarrow 
		\Z
	\]
	associated to the rank map $K\rightarrow \Z$. The second desired statement follows from considering the fiber sequence
	\[
		\BSL
		\rightarrow
		\BGL
		\rightarrow
		B_s\G_m
	\]
	associated to the determinant map $\BGL\rightarrow B_s\G_m$.
\end{proof}

\begin{theorem}[Theorem A]\label{thm:A}
	For every prime $p$, there exists an $\E_3$-motivic space map
	\[
		\Omega_{\G_m}\BGL
		\rightarrow 
		\BGL_1(\MGL)
	\]
	whose associated $\E_3$-$\MGL$-motivic Thom spectrum is isomorphic to $\mathrm{M}\F_p$. Precomposition by $\Omega_{\G_m}\BGL \rightarrow \BGL_1(\MGL)$ yields a $\E_3$-motivic space map
	\[
		\Omega_{\G_m}\BSL
		\rightarrow
		\BGL_1(\MGL)
	\]
	whose associated $\E_3$-$\MGL$-motivic Thom spectrum is isomorphic to $\mathrm{M}\Z$.
\end{theorem}

\begin{proof}
	Recall there there exists an $\E_3$-space map $U\rightarrow \BGL_1(\MU)$ whose $\E_3$-$\MU$-Thom spectrum is isomorphic to $\F_2$ and furthermore precomposing with $SU \hookrightarrow U$ yields an $\E_3$-$\MU$-space map $SU\rightarrow \BGL_1(\MU)$ whose $\E_3$-$\MU$-Thom spectrum is isomorphic to $\Z$ \cite[Theorem A]{cranfordGeneralizedThomSpectra2026}. These two statements are equivalent to the statemetns that there exists equivalences
	\[
		\MU \otimes_{\MU[\Omega U]} \MU \rightarrow \F_p
	\]
	and
	\[
		\MU \otimes_{\MU[\Omega SU]} \MU \rightarrow \Z
	\]
	of $\E_3$-$\MU$-algebras. Note that $\Omega U$ lifts to the motivic space $K$ and $\Omega SU=BU$ lifts to the motivic space $\BGL$. Both of these motivic spaces satisfy the hypothesis of Lemmas \ref{lem:realize} and \ref{lem:Tor_even} therefore we get equivalences
	\[
		\MGL \otimes_{\MGL[K]} \MGL \rightarrow \mathrm{M}\F_p
	\]
	and
	\[
		\MGL
		\otimes_{\MGL[\BGL]} \MGL \rightarrow \Z
	\]
	of $\E_3$-$\MGL$-motivic Thom spectra. These statemnts are equivalent to the claims of existance of an $\E_3$-motivic space map
	\[
		B_s K
		\rightarrow 
		\BGL_1(\MGL)
	\]
	whose associated $\E_3$-$\MGL$-motivic Thom spectrum is isomorphic to $\mathrm{M}\F_p$ and the existance of a $\E_3$-motivic space map
	\[
		B_s\BGL
		\rightarrow
		\BGL_1(\MGL)
	\]
	whose associated $\E_3$-$\MGL$-motivic Thom spectrum is isomorphic to $\mathrm{M}\Z$. The theorem follows by using Lemma \ref{lem:loop_spaces} to rewrite $B_sK$ and $B_s\BGL$.
\end{proof}

\begin{corollary}[Corollary \ref{intro:HHM}]\label{cor:HHM}
	Let $a_i$ denote the generators of the Lazard ring $(\MGL_{(\ast, \ast)})_0$ Then there is an equivalence
	\[
		\MGL/(a_1, a_2, \dots) \simeq \mathrm{M}\Z
	\]
	of $\E_3$-motivic spectra.
\end{corollary}

\begin{proof}
	The map comes from the surjection $\MGL[\BGL]\rightarrow \MGL$ that appeared from the equivalence
	\[
		\MGL
		\otimes_{\MGL[\BGL]} \MGL \rightarrow \Z
	\]
	of $\E_3$-$\MGL$-motivic Thom spectra that appeared in the proof of Theorem \label{thm:A}.
\end{proof}

\section{Normed Motivic Hochschild Homology}\label{sec:MHH}
In this section we prove Theorem B. We begin be reviewing motivic Hochschild homology.

\begin{definition}[Section 2.1 and Definition 4.1 of \cite{bachmannMotivicHochschildHomology2026}]
	Let $X \in \Spc(X)$ and $A\rightarrow E$ a normed motivic spectrum map. Then $X \odot^{A} E$ is defined to be the tensor $X \otimes E$ in $\NAlg(\Mod_A(\SH(\C)))=\NAlg(\SH(\C))_{/A}$.
\end{definition}

\begin{example}
	Consider the normed motivic spectrum maps $\MGL\rightarrow \Z$ of \cite[Example 16.34]{bachmannNormsMotivicHomotopy2021}, $\Z\rightarrow \F_2$ of \cite[Remark 14.6]{bachmannNormsMotivicHomotopy2021}, and $\MGL\rightarrow \F_2$ formed by composition. We can thus consider $\G_m \otimes^{\MGL} \F_2$ and $\G_m \otimes^{\MGL} \Z$.
\end{example}

To prove Theorem B, we will prove the existance of an $\E_2$-$\mathrm{M}\F_2$-algebra map $f_{\F_2}\colon \mathrm{M}\F_2[\BGL]\rightarrow \G_m \odot^{\MGL}\mathrm{M}\F_2$ and of an $\E_2$-$\mathrm{Z}\Z$-algebra map. $f_{\Z}\colon\mathrm{M}\Z[\BSL]\rightarrow \G_m \odot^{\MGL}\mathrm{M}\Z$ before we prove that $f_{\Z_2}$ and $f_{\Z}$ are equivalences. In \cite[Section 4.2]{bachmannMotivicHochschildHomology2026}, the authors construct a map from the free $\E_1$-$\mathrm{M}\F_2$-algebra $\mathrm{M}\F_2[\mu_{2,1}]$ on a generator $\mu_{2,1}$ in bidegree $(2,1)$ in the following way: to the class $\tau_0 \in \pi_{1,0}(\mathrm{M}\F_2 \otimes \mathrm{M}\F_2)$, they produce a class $\mu_{2,1} \in \pi_{2,1}(\G_m \odot \mathrm{M}\F_2)$ via the map $\Sigma^{1,1} \mathrm{M}\F_2 \otimes \mathrm{M}\F_2\rightarrow \G_m \odot \mathrm{M}\F_2$. The class $\mu_{2,1} \in \pi_{2,1}(\G_m \odot \mathrm{M}\F_2)$ determines an $\E_1$-$\mathrm{M}\F_2$-algebra map $\mathrm{M}\F_2[\mu_{2,1}]\rightarrow \G_m \odot \mathrm{M}\F_2$ by the universal property of $\mathrm{M}\F_2[\mu_{2,1}]$.

We will form our $\E_2$-algebra maps $f_{\F_2}$ and $f_{\Z}$ by a process that is significantly different from the way that the $\E_1$-$\mathrm{M}\F_2$-algebra map $\mathrm{M}\F_2[\mu_{2,1}]\rightarrow \G_m \odot \mathrm{M}\F_2$ above was formed. For context, we recall that \cite[Section 4]{bachmannMotivicHochschildHomology2026} first form their map then determine upon Betti realization = inversion of $\tau$ that they recover B\:oksdedt periodicity. To produce $f_{\F_2}$ and $f_{\Z}$, we will ``do things the other way around''. We will start with desired underlying topological equivalences $\THH(\F_2/\MU)\simeq \F_2[BU]$ of $\E_2$-$\F_p$-algebras and $\THH(\Z/\MU)\simeq \Z[BSU]$ of $\E_2$-$\Z$-algebras and deduce that there exist unique choices of $f_{\F_2}$ and $f_{\Z}$ that motivically enhance these topological equivalences. The uniqueness of such choices will come from celluluarity and Tateness forcing fully faithfullness of Betti realization. The precise content of this idea is stated in Proposition \ref{prop:mod_tau_tate}.

Before we state Proposition \ref{prop:mod_tau_tate}, we make our setting modulo $\tau$. 

\begin{notation}
We define $\overline{\mathrm{M}\F_2}\coloneqq \mathrm{M}\F_2/\tau$ and $\overline{\mathrm{M}\Z_2^{\wedge}}\coloneqq \mathrm{M}\Z/\tau$.
\end{notation}

\begin{theorem}[Theorem 2.9 of \cite{bachmannMotivicHochschildHomology2026}, forthcoming work of Bachmann-Hopkins]\label{thm:mod_tau_norm}
The motivic spectra $\overline{\mathrm{M}\F_2}$ and $\overline{\mathrm{M}\Z_2^{\wedge}}$ admit normed refinements.
\end{theorem}

The equivalence in Theorem A of $\F_p$ being Thom spectrum of the $\E_3$-map $U\rightarrow \BGL_1(\MU)$ and Proposition \ref{prop:thh_of_thom_bialg} give equivalences
\begin{lemma}
	There exist equivalence
	\[
		\left(\G_m \odot^{\MGL} \mathrm{M}\F_2\right)
		\otimes_{\MGL} \overline
		{\mathrm{M}\F_2}
		\simeq
		\G_m
		\odot^{\overline{\mathrm{M}\F_2}} \overline{\mathrm{M}\F_2} \otimes_{\MGL} \mathrm{M}\F_2
	\]
	of normed motivic spectra. Furthermore there exists an equivalence
	\[
		\left(\G_m \odot^{\MGL} \mathrm{M}\F_2\right)
		\otimes_{\mathrm{M}\F-2} \overline{\mathrm{M}\F_2}
		\simeq
		\G_m
	\odot^{\overline{\mathrm{M}\F_2}} \overline{\mathrm{M}\F_2} \otimes_{\MGL} \mathrm{M}\F_2 \otimes_{\overline{\mathrm{M}\F_2} \otimes_{\MGL} \mathrm{M}\F_2}\overline{\mathrm{M}\F_2}
	\]
	of normed motivic spectra.
\end{lemma}

\begin{proof}
	This follows from standard properties of $\G_m \odot {-}$ and of normed spectra.
\end{proof}

To work modulo $\tau$, we tensor with $\overline{\mathrm{M}\F_2}$ so that we seek to produce an $\E_2$-$\overline{\mathrm{M}\F_2}$-algebra map
\[
	\overline{\mathrm{M}\F_2} [\BGL]
	\rightarrow
	\G_m \odot^{\overline{\mathrm{M}\F_2}} \mathrm{M}\F_2 \otimes_{\mathrm{M}\F_2} \overline{\mathrm{M}\F_2}
\]

\begin{recollection} We recall several facts from \cite{bachmannMotivicHochschildHomology2026}. 
	The objects of $\NAlg(\SH(\C))_{\overline{\mathrm{M}\F_2}}$ have Voevodsky/topological operations $Q_i$ and higher divided power operations $\delta_i$. We use the indexing convetions established in \cite[Section 3]{bachmannMotivicHochschildHomology2026}.

There is a sequence of forgetful functors 
	\[
		\NAlg(\SH(\C))
		\rightarrow
		\NA
		\rightarrow
		\DA
		\rightarrow
		\Vect^{\Z \times \Z}_{\F_2}
	\]
studied in \cite[Section 4]{bachmannMotivicHochschildHomology2026} with the forgetful functor $\NAlg(\SH(\C)) \rightarrow \NA$ is given by forgetting the operations $Q_i$ for $i>2$, the forgetful functor $\NA \rightarrow \DA$ is given by forgetting the operations $Q_2$ and $Q_1$, and the forgetful functor $\DA \rightarrow\Vect^{\Z \times \Z}_{\F_2}$ is given by forgetting all operations thus is given by taking underlying bigraded $\F_2$-vector space.
\end{recollection}

\begin{lemma}
	We describe the $\mathrm{M}\F_2$-module $\overline{\mathrm{M}\F_2} \otimes_{\MGL}\mathrm{M}\F_2$ as an object of $\NA$. We have
	\[
		\pi_{\ast, \ast}^{\NA}\left(\overline{\mathrm{M}\F_2} \otimes_{\MGL}\mathrm{M}\F_2\right)
		\cong
		\Lambda_{\F_2}(\sigma a_0, \sigma a_1, \dots)
	\]
	with $\deg(\sigma a_i)=(2i+1,1)$ for all $i\geq 0$. Furthermore we have $Q_1=0$ and that $Q_2(e_i)=e_{2i+1}$ for all $i\geq 0$.
\end{lemma}

\begin{proof}
	By the Hopkins-Morel-Hoyois theorem, $\F_2$ is a regular quotient of $\MGL$, thus the $\pi_{\ast, \ast}^{\NA}$ calculation can be seen by resolving $\F_2$ as a $\MGL$-module by an appropiate Kozsul resolution. We have $Q_1=0$ for degree reasons so it remains to justify the $Q_2$ calculation.

	To see the $Q_2$ calculation, we first recall from \cite{kochmanHomologyClassicalGroups1971} that on Betti realiazation we have $Q_2$ on $\F_2[BU]$ is given by $Q_1(e_i)=e_{2i+1}$ mod decomposables. By naturality of power operations we conclude the same for the $\MGL$-dual steenrod algebra, and once again by naturality of power operations, we conclude the same modulo $\tau$. Finally we observe there cannot be decomposable for bidegree reasons.
\end{proof}

\begin{recollection}
	We recall the following from \cite{bachmannMotivicHochschildHomology2026}.
	\begin{itemize}
		\item There exists a derivation $d$ of bidegree $(1,1)$ coming from $\G_m$-suspension.
		\item There is functor $G:\NA\rightarrow \DA$ such that
			\[
				G(F(W)) 
				\simeq
				UF(W)
				\otimes
				F_{\DA}(dW)
			\]
		\item There if a functor $DG':D(\NA)\rightarrow D(\DA)$ given by first taking simplicial free $\NA$-resolution then applying $B$.
	\end{itemize}
\end{recollection}

\begin{lemma}\label{lem:SS}
	There exists a strongly convergent spectral sequence
	\[
		\pi_s\left(
			\F_2 \otimes_{U\Lambda_{\F_2}(\sigma a_0, \sigma a_1, \dots)}
			DG'(\Lambda_{\F_2}(\sigma a_0, \sigma a_1, \dots)
		\right)_{t,w}
		\Rightarrow
		\pi_{s+t, w}\left(
			(\G_m \odot^{\MGL} \F_2)\tau
			\right)
	\]
	There is a similar satement for $\Z$ mod $(2,\tau)$.
\end{lemma}

\begin{proof}
	This follows from \cite[Section 4.5]{bachmannMotivicHochschildHomology2026}
\end{proof}

Note that these there is a decomposition
\[
	\Lambda_{\F_2}(\sigma a_0, \sigma a_1) = \bigotimes_n \mathrm{Free}_{Q_2}(\sigma a_{4n})
\]

for $\mathrm{M}\F_2$. For $\mathrm{M}\Z$, there is one different in that the $ \mathrm{Free}_{Q_2}(\sigma a_{2n})$ starts with $\sigma a_1$ instead of $\sigma a_0$, so we get the decompositon

\[
	\Lambda_{\Z}(\sigma a_1) = \mathrm{Free}_{Q_2}(\sigma a_{2}) \otimes \bigotimes_{n>0} u\mathrm{Free}_{Q_2}(\sigma a_{4n})
\]

We can further describe $\mathrm{Free}_{Q_2}(\sigma a_{2n})$ following \cite[Remark 4.8]{bachmannMotivicHochschildHomology2026}. The free object $F(\mu_{m,n})$ consists of all the $Q_2^n \mu_{m,n}$ as well as al lthe $Q_1 Q_2^n (\mu_{m,n})$. This suggests $\mathrm{Free}_{Q_2}(\sigma a_{2n})$ should be thought of as $F(\sigma a_{2n})/(Q_1(\sigma a_{2n}))$ which guides us in making our derived models.

\begin{lemma}
	There exists a free $\NA$-resolution of $\mathrm{Free}_{Q_2}(\sigma a_{2n})$ given by the following complex: let there be exactly one module $\F_2\{x_m\}$ in each degree $(2^{m+1}n-1-m, 2^mn-1)$, which one should think of as $(Q_2)^m \sigma a_{2n}$ and define the differnetal on $x_m$ as $Q_1x_{m-1}$.
\end{lemma}

\begin{proof}
	By exactness, this complex preserves all $(Q_2)^m \sigma a_{2n}$ while killing $Q_1(\sigma a_{2n})$ thus producing a resolution of  $F(\sigma a_{2n})/(Q_1(\sigma a_{2n}))$.
\end{proof}
Now that we have made a free $NA$-resolution of $F(\sigma a_{2n})/(Q_1(\sigma a_{2n}))$, an object of $D(\NA)$, we have made some progress towards understanding the object 
\[
\F_2 \otimes_{U\Lambda_{\F_2}(\sigma a_0, \sigma a_1, \dots)}(DG'(F(\sigma a_{2n})/(Q_1(\sigma a_{2n}))))
\]

in $D(\DA)$. Recall that $DG'$ is defined by applying $G$ to the free $\NA$-resolution we calculated above. We begin with calculating the homology associated to it as this calculation is determined by the operations.
\begin{lemma}\label{lem:H_of_DG}
	The trigraded homotopy of
	\[
	\F_2 \otimes_{U\Lambda_{\F_2}(\sigma a_0, \sigma a_1, \dots)}(DG'(F(\sigma a_{2n})/(Q_1(\sigma a_{2n}))))
	\]
	is
	\[
		\F_2[d x_0]
		\otimes
		\otimes_{k \geq 0}
		\Gamma_{\F_2} y_{n,k}
	\]
	with $\abs{d x_0}=(0,2n,n)$ and $\abs{y_{n,k}}=(1, 2^{k+2}n-1, 2^{k+1}n)$.
\end{lemma}

\begin{proof}
	Firstly we explain the effects of $G$: the $x_k$'s become replaced with $d x_k$ where $\abs{d x_k}=(2^{k+1}n-k, 2^{k}n)$ becuase $d$ is of bidgree $(-1, -1)$. Furthermore the diffferentials become $\delta(d x_k)=Q_{-1}x_{k-1}$ $d Q_1=Q_{-1}d$. Note that $\delta_0=0$ so $H_0$ is just made up of the Chow degree $0$ elements as desired. To see the divded power algebra parts, one argues in a similar way to the proof of \cite[Theorem 3.19]{bachmannMotivicHochschildHomology2026},
\end{proof}

Now that we calculated the homology of
\[
	\mathrm{M}\F_2 
	\otimes_{U\left(F(\sigma a_{2n})/(Q_1(\sigma a_{2n}))\right)}
	(DG')F(\sigma a_{2n})/(Q_1(\sigma a_{2n}))
\]
we are ready to realize it in Lemma \ref{lem:DG'}.
\begin{lemma}\label{lem:DG'}
	The object
	\[
		\mathrm{M}\F_2 
		\otimes_{U\left(F(\sigma a_{2n})/(Q_1(\sigma a_{2n}))\right)}
		(DG')F(\sigma a_{2n})/(Q_1(\sigma a_{2n}))
	\]
	of $D(\DA)$ is equivalent to
	\[
		\mathrm{M}\F_2[d x_0]
		\otimes
		\operatorname{Bar}_{\ast}\left(
			\mathrm{M}\F_2,
			\Lambda_{\F_2}(y_0, y_1, \dots),
			\mathrm{M}\F_2
		\right)
	\]
	with $\deg(d x_0)=(2n, n)$ and $\deg(y_k)=(2^{k+2}m-1, 2^{k+1}n)$ for all $k\geq 0$. Note that we supressed the third grading coming from Bar.
\end{lemma}

\begin{proposition}\label{prop:mod_tau_tate}
	The $\overline{\mathrm{M}\F_2}$-modules $\left(\G_m \odot^{\MGL} \F_2\right)/\tau$ and $\left(\G_m \odot^{\MGL} \mathrm{M}\Z\right)/(2, \tau)$ are Tate.
\end{proposition}

\begin{proof}
	Consider the spectral sequence of Lemma \ref{lem:SS}. By evenness and tridegree reasons, the spectral sequence collapses on the $\E_2$-page with no extension problems, therefore it suffuces to show Tateness and cellularity 
\[
	\F_2 \otimes_{U\Lambda_{\F_2}(\sigma a_0, \sigma a_1, \dots)}(DG'(\Lambda_{\F_2}(\sigma a_0, \sigma a_1, \dots).
\]
Note that this split up, via the $\Lambda_{\F_2}(\sigma a_0, \sigma a_1, \dots)$ as a tensor product of the $F(\sigma a_{2n})/(Q_1(\sigma a_{2n}))$'s. Tensor product preserves Tateness and cellularity, and so do the procesees of $DG'$ therefore it suffices to show Tateness and cellularity for the individual
\[
	\mathrm{M}\F_2 
	\otimes_{U\left(F(\sigma a_{2n})/(Q_1(\sigma a_{2n}))\right)}
	(DG')F(\sigma a_{2n})/(Q_1(\sigma a_{2n}))'s.
\]
We showed this in Lemma \ref{lem:DG'}. The $\Z$ case follows from the minor variation previously mentioned.
\end{proof}

Now that we've completed the hard work of showing the Tateness and cellularity of the $\overline{\mathrm{M}\F_2}$-modules $(\G_m \odot^{\MGL}\F_2)\tau$ and $(\G_m \odot^{\MGL}\Z)/(2, \tau)$, we move on to showing that we have Tateness and cellularity before modding out $\tau$. 

\begin{lemma}
There are $\E_2$-$\MU$-algebra equivalence
\[
	\mathrm{Re}(\G_m \odot^{\MGL_p^{\wedge}}\mathrm{M}\F_2[\tau^{-1}])
	\simeq
	\THH(\F_2/\MU)
	\rightarrow
	\F_2[\BU]
	\simeq
	\mathrm{Re}(\mathrm{M}\F_2[\BGL])
\]
\end{lemma}

\begin{lemma}
	The Tate $\overline{\mathrm{M}\F_2}$-modules $\overline{\mathrm{M}\F_2}$-module $\G_m \odot^{\MGL}\F_2/\tau$ and $(\G_m \odot^{\MGL}\Z)/(2,\tau)$ are pure Tate. In paticular
	\[
		\G_m \odot^{\MGL}\F_2
		\simeq
		\oplus_{n \geq 0} \left(\Sigma^{2n, n}\mathrm{M}\F_2\right)^{\oplus p(n)}
	\]
	and
	\[
		(\G_m \odot^{\MGL}\Z)/2
		\simeq
		\oplus_{n \geq 0} \left(\Sigma^{2n, n}\mathrm{M}\F_2\right)^{\oplus q(n)}
	\]
	each as $\mathrm{M}\F_2$-modules where $p(n)$ denotes the partition function and $q(n)$ denotes the reduced partition function $p(n)-p(n-1)$.
\end{lemma}
\begin{proof}
	We've already established Tateness to it remains to justify the pureness and the ranks $p(n)$ and $q(n)$. This follows from a counting arguement given Lemma \ref{lem:H_of_DG}.
\end{proof}

\begin{lemma}\label{lem:tate_with_tau}
	The $\mathrm{M}\F_2$-modules $\left(\G_m \odot^{\MGL} \F_2\right)$ and $\left(\G_m \odot^{\MGL} \mathrm{M}\Z\right)/2$ are pure Tate.
\end{lemma}

\begin{proof}
	We first argue that $\pi_{s,t}=0$ for these two $\mathrm{M}\F_2$-modules whenever $s<t$. This follows from $\MGL$, $\mathrm{M}\Z$, $\mathrm{M}\F_2$ being very effective (\cite[Section 5]{spitzweckMotivicTwistedKtheory2012}, \cite[Section 7]{hoyoisAlgebraicCobordismMotivic2015}) and that being very effective is preserved by the various operations we apply. Do decude purity and the $\mathrm{M}\F_2$-module structures, this from looking at the cofiber sequences assoicated to $\tau$ and using the vanishing above argued from very effectiveness.
\end{proof}

\begin{corollary}
	The motivic spectra $\MGL[BGL]$ and $\MGL[BSL]$ are pure tate of finite type and have the same $\mathrm{M}\F_2$-module decomposition as $\G_m \odot^{\MGL}\F_2$ and $(\G_m \odot^{\MGL}\Z)$, respectively.
\end{corollary}

\begin{lemma}\label{lem:betti_realize}
	Betti realization is fully faithful on nonnegative, finite-type pure Tate $\mathrm{M}\F_2$-modules, and a similar statement is true for $\mathrm{M}\Z$-modules
\end{lemma}

\begin{proof}
	To show the claim $\mathrm{M}\F_2$-modules, by our assumptions it suffices to show this claim for the $\mathrm{M}\F_2$-modules $\sigma^{n, 2n}\mathrm{M}\F_2$. This case was proven in \cite[middle of page 6]{bachmannMotivicHochschildHomology2026}. The arguement for $\mathrm{M}\Z$-modules is essentially the same starting with the fact that $\pi_{\ast, \ast}(\mathrm{M}\Z)\cong \Z[\tau]$.
\end{proof}

\begin{theorem}[Theorem B]\label{thm:B}
	There exists an equivalence
	\[
		\G_m \odot^{\MGL}\mathrm{M}\F_2
		\simeq
		\mathrm{M}\F_2[\BGL]
	\]
	of $\E_2$-$\mathrm{M}\F_2$-algberas. Furthermore there exists an equivalence
	\[
		\G_m \odot^{\MGL}\mathrm{M}\Z
		\simeq
		\mathrm{M}\Z[\BSL]
	\]
	of $\E_2$-$\mathrm{M}\Z$-algberas.
\end{theorem}

\begin{proof}
	We use Lemma \ref{lem:betti_realize} to motivially lift the equivalences $\THH(\F_2/\MU)\cong \F_2[\BU]$ and $\THH(\Z/\MU) \cong \Z[\BSU]$ of $\E_2$-rings from topology by Lemma \ref{lem:betti_realize}.
\end{proof}

\bibliographystyle{amsalpha}
\bibliography{ref}
\medskip
\noindent
\textsc{Department of Mathematics, Northwestern University}\\
{\ttfamily prestoncranford2027@u.northwestern.edu}
\end{document}